\documentclass[a4paper,12pt,tbtags,leqno]{amsart}

\usepackage{amsfonts,amscd,amssymb,amsmath}

\newtheorem{thm}{Theorem}[section]

\numberwithin{equation}{section}

\date{}

\begin{document}

\title{On Summation of $p$-Adic  Series}

\author{Branko Dragovich}   
\address{$^1$Institute of Physics, University of Belgrade,   Belgrade, Serbia }
\address{$^2$Mathematical Institute, Serbian Academy of Sciences and Arts, Belgrade, Serbia }
\email{dragovich@ipb.ac.rs}



\begin{abstract}
Summation of the $p$-adic functional series $\sum \varepsilon^n \, n! \, P_k^\varepsilon (n; x)\, x^n ,$ where
$P_k^\varepsilon (n; x)$ is a polynomial in $x$ and $n$ with rational coefficients, and $\varepsilon = \pm 1$, is considered. The series is convergent
in the domain $|x|_p \leq 1$ for all primes $p$. It is found the general form of polynomials $P_k^\varepsilon (n; x)$ which
provide rational sums when $x \in \mathbb{Z}$.  A class of generating polynomials $A_k^\varepsilon (n; x)$  plays a central role in the summation procedure. These generating polynomials  are  related to many sequences of integers. This is a brief review with some new results.

\end{abstract}

\maketitle

\section{Introduction}

In the last thirty years $p$-adic analysis has had successful applications in $p$-adic mathematical physics (from strings to complex systems
and the universe as a whole) and in some related fields (in particular in bioinformation systems, see, e.g. \cite{bd01}), see \cite{freund,vvz} for an early review and \cite{bd02} for a recent one. The $p$-adic series, as a part of
$p$-adic analysis, have been also considered as mathematical tools of potential applications in physics. The main reason was in
the following. There are series, in particular the power series, of the same form in physics and $p$-adic analysis. Recall that the results of
physical measurements are rational numbers which belong to all $p$-adic number  fields $\mathbb{Q}_p$ as well as to $\mathbb{R}$. Hence,
there is a sense to consider $p$-adic series in rational points, i.e. when argument and sum are some rational numbers, and look for their  possible
physical content. One of motivations was related to the convergence of  series. Namely, the power series $\sum a_n \, x^n ,$ where $a_n \in \mathbb{Q}$,  can be treated simultaneously as $p$-adic or as real, depending on values of the argument $x$. Many series which are divergent in the real case are convergent
in the $p$-adic one. There are many examples of divergent series, in particular, in perturbation expansions of quantum field theory and string theory,
which diverge due to factorials. This problem of divergences in some real series motivated an investigation of various series with factorials with respect to
the $p$-adic norm. To this end,  many $p$-adic convergent series have been constructed and found their rational sums for some rational arguments.

In this paper we are interested in $p$-adic invariant summation of a class of infinite functional series which terms contain $n!$, i.e.
$\sum \varepsilon^n \, n!\, P_k^\varepsilon (n; x)\, x^n ,$ where $\varepsilon = \pm 1 $ and  $P_k^\varepsilon (n; x)$ are polynomials in $x$  and $n$ of degree $k $. We  show that there exist
polynomials $P_k^\varepsilon (n; x)$ for any degree $k ,$ such that for any $x \in \mathbb{Z}$ the corresponding sums are also  integer or rational numbers. Moreover, we have found recurrence relations  to calculate all ingredients of such  $P_k^\varepsilon (n; x)$.

 All necessary general information on $p$-adic series can be found in standard books on $p$-adic analysis, see, e.g. \cite{schikhof}.

\section{Series and $p$-adic invariant summation in integer points}

We will give now a brief review of some $p$-adic series with factorials presented in the papers
\cite{bd1,bd2,bd3,bd4,bd5,bd6,bd7,bd8,bd9,bd10,bd11}.

Recall that necessary and sufficient condition for the $p$-adic power series
to be convergent  is that the general term vanishes as degree goes to infinity  \cite{schikhof,vvz},  i.e.
\begin{align}
S(x) = \sum_{n=1}^{+\infty} a_n x^n ,   \quad a_n \in  \mathbb{Q}_p ,
\quad x \in \mathbb{Q}_p ,  \quad |a_n x^n|_p \to 0 \,\, \text{as} \,\, n
\to \infty ,     \label{2.1}
\end{align}
where $|\cdot|_p$ denotes the $p$-adic absolute value (also called $p$-adic norm). As in the real case,
convergent infinite $p$-adic series can be summed using usual rules as for the finite series.

We are interested here in the series which are simultaneously convergent with respect to $p$-adic absolute value for
every prime number $p$. One of the simplest such examples is the well known series
\begin{align}
\sum_{n=0}^\infty n! = 0! + 1! + 2! + ... + n! + ... , \, \quad \, \, |n!|_p = p^{-\frac{n -s_n}{p-1}} \to 0 \, \, (n \to \infty) ,      \label{2.2}
\end{align}
where $ s_n = n_0 + n_1 + ... + n_r $ is the sum of digits in the canonical expansion of  number $n$ in base $p$, i.e. $n = n_0 + n_1 p + ... + n_r p^r  .$
It can be shown that if the sum of series \eqref{2.2} has a rational value then it cannot be the same for all
$p$-adic cases. Probably the sum of series \eqref{2.2} is  irrational number  in all $\mathbb{Q}_p$.

From an application point of view it is more interesting to investigate $p$-adic series which sum is a rational number, the same for all primes $p$.
As a simple illustrative example one can  refer to
 \begin{align}
\sum_{n=1}^\infty n!\ n =  1!\ 1 + 2!\ 2 + ... + n! \ n + ... = - 1 \, .       \label{2.3}
\end{align}
In the proof of \eqref{2.3} one can employ any of the following two properties:
\begin{align}
  (i): \, \,  n!\ n = (n+1)! - n! \,, \quad \, \qquad  (ii): \, \,  \sum_{n = 1}^{N-1} n! \, n  = -1 + N! \,.   \label{2.4}
\end{align}

In this section we present a review of some convergent $p$-adic series, where summation is performed using mainly property like  $(i)$ in
\eqref{2.4}. Method similar to $(ii)$ of \eqref{2.4} will be employed in the next section.

In \cite{bd1} was shown that the ground-state energy perturbation series of the anharmonic oscillator
\begin{align}
E (\lambda)  = \frac{1}{2} + \sum_{n=1}^\infty (-1)^{n+1} a_n 2^{-2n} \lambda^n \, ,  \quad  a_n \in \mathbb{N} \, , \label{2.5}
\end{align}
which is  divergent in  the real case ($a_1 =3, a_2 = 2\cdot 3\cdot 7,  a_3 = 2^2 \cdot 3^2 \cdot 37, ... $), is $p$-adic convergent for $|\lambda|_p < 1$ if $p \neq 2$ and $|\lambda|_2 < \frac{1}{4}$.
Paper \cite{bd2} contains two similar examples to \eqref{2.5} and the following one:
\begin{align}
F (x) = \sum_{n=1}^\infty (-1)^{n+1}\, (n!)^k \, (n^\ell + w)^m \, x^n \, , \quad w \in \mathbb{Z} , \, k \in \mathbb{N} ,
\, m \in \mathbb{N}\cup \{ 0 \} \, ,  \label{2.6}
\end{align}
which is convergent for every $p$ in the range $|x|_p \leq 1$ and consequently includes all integers.  After suitable rearrangement of the series terms  for $F (x)$ and $-F(x)$, and then taking $F(x)+ (-F(x)) = 0 $,
one obtains  formula
\begin{align}
\sum_{n=1}^\infty (n!)^k \{(n+1)^k [(n+1)^\ell + w]^m \, x - (n^\ell + w)^m  \} \, x^{n-1}  = -(1 + w)^m   .     \label{2.7}
\end{align}
One can easily see that \eqref{2.7} contains \eqref{2.3} as a particular case.

In \cite{bd2} was also discussed possible connection between integer sums of $p$-adic series convergent in $|x|_p \leq 1$ for all $p$ and its
real counterpart which is divergent. Recall that the sum of the divergent series depends on the way how summation is performed. Using the same summation
procedure as in $p$-adic convergent cases will give the same integer value in the real divergent counterpart. This kind of summation of divergent series
was called adelic summation.   For example, series \eqref{2.3}  is highly
divergent from real point of view, but according to the adelic summation the corresponding real sum should be also $-1$. Perhaps this  adelic summation
will find an application in physics and related sciences.

 Inspired by  divergent series in a zero-dimensional model of quantum field theory, in paper \cite{bd4} the following $p$-adic series was considered
\begin{align}
F (x) = \sum_{n=1}^\infty (-1)^{n+1} \prod_{i=1}^\alpha((\mu_i n + \nu_i )! )^{k_i} P_\ell (n) \, x^n \,,  \label{2.8}
\end{align}
where $\alpha, \mu_i \in \mathbb{N}, \, \, \nu_i \in \mathbb{Z},  \, \mu_i + \nu_i \geq 1$ and $\ell , k_i \in \mathbb{N}_0$, and at least one of $k_i \geq 1$.
The domain of convergence contains $|x|_p \leq 1$ and hence includes $x \in \mathbb{Z} .$
By suitable rearrangement of terms for $F (x)$ and $-F(x)$, and then taking $F(x)+ (-F(x)) = 0 $, one obtains summation equality
\begin{align}
\sum_{n=1}^\infty \prod_{i=1}^\alpha ((\mu_i n + \nu_i )! )^{k_i} &\left[ \prod_{i=1}^\alpha (\mu_i n + \nu_i + 1)_{\mu_i}^{k_i} P_\ell (n +1) \, x  - P_\ell (n)    \right]\, x^{n-1} \nonumber \\ & = - \prod_{i=1}^\alpha((\mu_i  + \nu_i )! )^{k_i} P_\ell (1)\, ,    \label{2.9}
\end{align}
where $(\mu_i n + \nu_i + 1)_{\mu_i}^{k_i} = (\mu_i n + \nu_i + 1)^{k_i} (\mu_i n + \nu_i + 2)^{k_i}  ... (\mu_i n + \nu_i + \mu_i)^{k_i}$. Polynomial $P_\ell (n)$ is
\begin{align}
P_\ell (n) = c_\ell n^\ell + c_{\ell-1} n^{\ell -1} + ... + c_0 \, ,          \label{2.10}
\end{align}
where $c_0 , c_1,  ... , c_\ell$ are  some integers.
Taking $x = +1$ or $-1$ and various values of parameters in \eqref{2.9} one gets sum for plenty of $p$-adic numerical series.

Paper \cite{bd5} is devoted to rational summation of the wider class  of $p$-adic series
\begin{align}
S^\varepsilon (x) = \sum_{n=1}^\infty \varepsilon^{n} \prod_{i=1}^I((\mu_i n + \nu_i )! )^{\lambda_i} P_k (n) \, x^{\alpha n + \beta} \,,  \label{2.11}
\end{align}
where $\varepsilon = \pm 1, \, \, \, I, \mu_i, \alpha \in \mathbb{N}, \, \, \nu_i \in \mathbb{Z}, \, \mu_i + \nu_i \geq 1 , \, \, \, \lambda_i, \beta \in \mathbb{N}_0$ and at least one of $\lambda_i \in \mathbb{N} .$ \  \
 $P_k (n)$ is a polynomial
\begin{align}
P_k (n) = C_k n^k + C_{k-1} n^{k -1} + ... + C_0 \, ,          \label{2.12}
\end{align}
where $C_0 , C_1,  ... , C_k \in \mathbb{Z}$. The domain of convergence of the series \eqref{2.11} is
\begin{align}
|x|_p < p^{\frac{1}{(p-1)\alpha} \sum_{i =1}^I \mu_i \lambda_i}  \label{2.13}
\end{align}
which includes $x \in \mathbb{Z}$ for all $p$-adic norms.

 As rational summation  one has in mind summation of convergent series, like  \eqref{2.11},
that for a rational argument $x$ exists a rational sum  $S (x)$.  For the series \eqref{2.11}, rational summation makes a restriction on possible
polynomials $P_k (n)$.  To find suitable polynomials $P_k (n)$  it is useful to introduce an auxiliary polynomial in  \eqref{2.11}
\begin{align}
P_\ell (n) = a_\ell \, n^\ell + a_{\ell-1} \,  n^{\ell -1} + ... + a_0 \, ,          \label{2.14}
\end{align}
where $0 \leq \ell < k$ and  $a_0 , a_1,  ... , a_\ell$ are  some integer numbers. Using the same summation procedure as in the previous cases,
and replacing $P_k (n)$ by $A_\ell (n)$ in \eqref{2.11}, follows the   equality
\begin{align}
\sum_{n=1}^\infty  \varepsilon^n \prod_{i=1}^I ((\mu_i n + \nu_i )! )^{\lambda_i} &\left[ \prod_{i=1}^I (\mu_i n + \nu_i + 1)_{\mu_i}^{\lambda_i} A_\ell (n +1) \, x^\alpha  - \varepsilon A_\ell (n)    \right]\, x^{\alpha n + \beta} \nonumber \\ & = - \prod_{i=1}^I((\mu_i  + \nu_i )! )^{\lambda_i} A_\ell (1)\, x^{\alpha + \beta} ,    \label{2.15}
\end{align}
where $(\mu_i n + \nu_i + 1)_{\mu_i}^{\lambda_i} = (\mu_i n + \nu_i + 1)^{\lambda_i} (\mu_i n + \nu_i + 2)^{k_i}  ... (\mu_i n + \nu_i + \mu_i)^{\lambda_i}$. Summation expression \eqref{2.15} is valid in the domain of convergence of \eqref{2.11}, i.e. includes all $x \in \mathbb{Z} .$  Taking $x = t \in \mathbb{Q}$ which satisfies the domain of convergence \eqref{2.13} one can construct a polynomial
\begin{align}
P_k (n;  t) = \prod_{i=1}^I (\mu_i n + \nu_i + 1)_{\mu_i}^{\lambda_i} A_\ell (n +1) \, t^\alpha  - \varepsilon A_\ell (n)  \label{2.16}
\end{align}
where $k = \ell + \sum_{i=1}^I \mu_i \lambda_i$,  which depends on $t$ and leads to a rational sum of \eqref{2.11}.
For a given polynomial $P_k (n)$ the series \eqref{2.11} has a rational sum if there exists an auxiliary polynomial $A_\ell (n)$ such that expression \eqref{2.16} is satisfied.
When $t \in \mathbb{Z}$ then the corresponding sum is an integer the same in all $\mathbb{Z}_p$,
where $\mathbb{Z}_p$ is the ring of $p$-adic integers. In the next section we shall see how one can construct all related polynomials $P_k (n; x)$ for any degree $k$.

In \cite{bd6} is presented extension of summation formula \eqref{2.15} in the form
\begin{align}
\sum_{n=1}^\infty  \varepsilon^n \frac{\prod_{i=1}^I ((\mu_i n + \nu_i )! )^{\lambda_i}}{\prod_{j=1}^J ((\rho_j n + \sigma_j )! )^{\eta_j}}
&\left[ \frac{\prod_{i=1}^I (\mu_i n + \nu_i + 1)_{\mu_i}^{\lambda_i}}{\prod_{j=1}^J (\rho_j n + \sigma_j + 1)_{\rho_j}^{\eta_j}} A_\ell (n +1) \, x^\alpha  - \varepsilon A_\ell (n)    \right]\, x^{\alpha n + \beta} \nonumber \\ & = - \frac{\prod_{i=1}^I((\mu_i  + \nu_i )! )^{\lambda_i}}{\prod_{j=1}^J((\rho_j  + \eta_j )! )^{\eta_j}}  A_\ell (1)\, x^{\alpha + \beta} .    \label{2.17}
\end{align}
Paper \cite{bd7} contains an elaboration of rational summation considered in \cite{bd5}.

Papers \cite{bd8,bd9,bd10} were presented at conferences on $p$-adic functional analysis and published in their proceedings. A very simple case of \eqref{2.15}
is
\begin{align}
\sum_{n=0}^\infty n! \, [(n+1) A_\ell (n+1) - A_\ell (n)] = - A_\ell (0)  ,   \label{2.18}
\end{align}
where $A_\ell (n)$ is defined by \eqref{2.14}. It is obvious that all possible polynomials $A_\ell (n)$ generate the corresponding polynomials
$P_k (n) = (n+1) A_\ell (n+1) - A_\ell (n)$, with $k = \ell +1 \geq 1$, which give integer sum $- A_\ell (0)$, i.e.
\begin{align}
\sum_{n=0}^\infty n! \, P_k (n) = - A_\ell (0)  .   \label{2.19}
\end{align}
It is useful  to investigate the series of the form
\begin{align}
\sum_{n=0}^\infty n! \, [n^k  + u_k] = v_k  ,   \label{2.20}
\end{align}
where  $u_k , v_k$ are pairs of integers. The series \eqref{2.20} is analyzed in \cite{bd8}, where
\begin{align}
(n+1) A_{k-1} (n+1) - A_{k-1} (n) = n^k  + u_k   , \quad k = 1, 2, 3, ....   \label{2.21}
\end{align}
Note that \eqref{2.21} contains  a system of $k +1$ linear equations with $k +1$ unknowns $(a_0, a_1, ..., a_{k-1}, u_k)$   which has always solution, and
 that $v_k = - A_{k-1}(0)$. It was also shown that pairs of integers $u_k , v_k$ are unique. If we know integers $u_1 , u_2, ... , u_k$ and $v_1 , v_2, ... , v_k$
 then  one can write general summation formula
\begin{align}
\sum_{n=0}^\infty n! \, P_k (n) = Q_k  ,  \quad P_k (n) = C_k\, n^k  + ... + C_1\, n + C_0 , \quad  Q_k = \sum_{j=1}^k C_j \, v_j  ,  \label{2.22}
\end{align}
where $C_0 = \sum_{j=1}^k C_j \, u_j $  and    $C_0, C_1, ... , C_k \in \mathbb{Q}$. Hence only these polynomials $P_k (n)$ which are defined by \eqref{2.22}
give $p$-adic invariant rational summation, where nontrivial role plays term $C_0$.

In paper \cite{bd9} a modified generalized hypergeometric series was considered, the domain of convergence was found and a summation formula was derived.
It was also reconsidered summation formula \eqref{2.20} and related recurrence relation was found (see also \cite{bd10})
\begin{align}
S_{k+1} (n) =  \delta_{0k} - k S_k (n) - \sum_{\ell =0}^{k-1} \binom{k+1}{\ell}  S_\ell (n)  + n! n^k \, ,  \label{2.23}
\end{align}
where $S_k (n) = \sum_{i=0}^{n-1} i!\, i^k $ and $\delta_{0k}$ is the Kronecker  symbol ($\delta_{0k} =1$ if $k=0$ and $\delta_{0k} = 0$ if $k\neq 0$).
From \eqref{2.23} follows
\begin{align}
\sum_{i=0}^{n-1} i! \,(i^k + u_k) = v_k  + n! A_{k-1} (n)  \,,   \label{2.24}
\end{align}
where $u_k$ and $v_k$  satisfy the following recurrence relations:
\begin{align}
&u_{k+1} = -k u_k - \sum_{\ell =1}^{k-1} \binom{k+1}{\ell} u_\ell + 1 \, ,  \quad  u_1 = 0, \, \, k\geq 1  \, , \label{2.25} \\
&v_{k+1} = -k v_k - \sum_{\ell =1}^{k-1} \binom{k+1}{\ell}  v_\ell - \delta_{0k} \, ,  \quad k\geq 0 \, . \label{2.26}
\end{align}
Note that $A_k (n)$ is a polynomial in $n$ and in limit $n \to \infty$ formula \eqref{2.24} becomes
\eqref{2.20}. Using recurrence relations  \eqref{2.25} and \eqref{2.26} one can calculate $u_k$
and $v_k$, which the first eleven values are presented in Table 1.

Note that in \cite{bd9} is introduced the following functional summation formula:
\begin{align}
\sum_{i=0}^\infty n! \,[n^k x^k + U_k (x)]\, x^n = V_k (x)   \,,   \label{2.27}
\end{align}
where $U_k (x)$ and $V_k (x)$ are certain polynomials in $x$ with integer coefficients. When $x =1$ we have
$U_k (1)= u_k$ and $V_k (1) = v_k$.

Paper \cite{bd9}  contains a  proof that the sum of $p$-adic series
\begin{align}
\sum_{n=0}^\infty n! \, n^k \, x^n  \, , \quad  k \in \mathbb{N}_0 \, , \,\, x \in \mathbb{N}\setminus \{ 1 \}  \label{2.28}
\end{align}
cannot be the same rational number  in $\mathbb{Z}_p$ for every $p$. It also contains conjecture that the sum of \eqref{2.28} is
a rational number iff $k=x = 1$. In particular, according to this conjecture $u_k \neq 0$ in \eqref{2.20} when $k \neq 1$ and the corresponding sums $\sum n! n^k$ are $p$-adic irrational numbers. Following this research in \cite{murty} was shown that $u_k \neq 0$ for $k \equiv 0$ or  $2$ (mod 3). Then has shown \cite{alexander} that $u_k$ can be zero at most twice (see also \cite{subedi1,subedi2}).

\begin{table}
 \begin{center}
 {\begin{tabular}{|c|c|c|c|c|c|c|c|c|c|c|c|}
 \hline \ & \ & \ & \ & \ & \ & \ & \ & \ & \ & \ & \ \\
 k & 1  & 2 & 3 & 4 &  5 & 6 & 7 & 8 & 9 & 10 & 11 \\
  \hline \  & \  &  \ & \ & \ & \ & \ & \ & \ & \ & \ & \ \\
 $u_k$ &  0  & 1 & -1 & -2 & 9 & -9 & -50 & 267 & -413 & -2180 & 17731 \\
  \hline \  & \  & \  &  & \ & \ & \ & \ & \ & \ & \ & \ \\
 $v_k$ & -1 & 1 & 1 & -5 & 5 & 21 & -105 & 141 & 777 & -5513 & 13209 \\
 \hline
\end{tabular}}{}
\vspace{4mm}
\caption{ The first eleven values of $u_k$ and $v_k$ \cite{bd9,bd10}. Note that in encyclopedia of
integer sequences \cite{sloane}, sequence $u_k$ overlaps with  A000583 in absolute values, while
 $v_k$ coincides with A014619.}
\end{center}
\end{table}

Paper \cite{bd11} is devoted to  the first-order and second-order
 differential equations which contain as a solution an analytic function of the form
\begin{align}
F_k (x) = \sum_{n=0}^\infty n! \, P_k (n) \, x^n  \, , \quad  k \in \mathbb{N}_0 \, , \,\, |x|_p < p^{\frac{1}{p-1}} \, ,  \label{2.29}
\end{align}
where $P_k (n) = n^k + C_{k-1} n^{k-1} + ... + C_0$ is a polynomial in $n$ with $C_i \in \mathbb{Z}$ (in a more general case $C_i \in \mathbb{Q}$ or $C_i \in \mathbb{C}_p$). It is shown the existence of linear differential equations which solution is $F_k (x)$  for any polynomial $P_k (n)$ and some of such equations are constructed. For example, the related first-order and second-order linear differential  equations for $F_0 (x) = \sum_{n=0}^\infty n! \, x^n$ are:
\begin{align}
&x^2 \, F_0'(x) + (x-1)\, F_0 (x) = -1 \, ,  \label{2.30} \\
&x^2 \, F_0''(x) + (3x-1) \, F_0'(x) +  F_0(x) = 0 \, .  \label{2.31}
\end{align}

\section{Functional summation formula and  polynomials $A_k^\varepsilon (n; x)$}

This section is devoted to investigation of the finite functional summation formula of the form
\begin{align}
\sum_{i=0}^{n-1} \varepsilon^i i! \,[i^k x^k + U_k^\varepsilon (x)]\, x^i = V_k^\varepsilon(x)  + n!\, A_{k-1}^\varepsilon (n; x)  \,,   \label{3.1}
\end{align}
where $\varepsilon = \pm 1$ and  $A_k^\varepsilon (n; x)$ are certain polynomials in $x$, which coefficients are polynomials in $n$ (also of degree $k$) with integer coefficients.  $U_k^\varepsilon (x)$ and $V_k^\varepsilon (x)$ are polynomials related to the  polynomials $A_{k-1}^\varepsilon (n; x)$ as follows:
\begin{align}
U_k^\varepsilon (x) = x\, A_{k-1}^\varepsilon (1; x) - \varepsilon \, A_{k-1}^\varepsilon (0; x) \, ,  \quad V_k^\varepsilon (x) = - \varepsilon A_{k-1}^\varepsilon (0; x) . \label{3.2}
\end{align}
If $x =1$ and $\varepsilon = + 1$ then summation formula \eqref{3.1} reduces to \eqref{2.24}. When $\varepsilon = + 1$ and  $n \to \infty $ formula \eqref{3.1} evidently  becomes
\eqref{2.27}.

This section is based on papers \cite{bd12,bd13} and their elaboration with some new results. Here the main role is plaid by the polynomials $A_k^\varepsilon (n; x)$, because they contain all information on properties of summation formula \eqref{3.1} .

\begin{thm}   \label{Th1}
Let $ S_k^\varepsilon (n; x) = \sum_{i=0}^{n-1} \varepsilon^i \, i! \, i^k \, x^i $, where $\varepsilon = \pm 1$. Then one has the recurrence formula
\begin{align}
S_k^\varepsilon (n; x) = \delta_{0k} + \varepsilon \, x  \,  S_{0}^\varepsilon (n; x) + \varepsilon \, x  \, \sum_{\ell =1}^{k+1} \binom{k+1}{\ell}\, S_{\ell}^\varepsilon (n; x) - \varepsilon^n \, n! \, n^k \, x^n \,.  \label{3.3}
\end{align}
\end{thm}

\begin{thm}
\label{Th2}
The recurrence relation \eqref{3.3} has solution in the form
\begin{align}
\sum_{i=0}^{n-1} \varepsilon^i \, i!\, [i^k \, x^{k} \, + U_{k}^\varepsilon (x)] \, x^{i}
= V_{k-1}^\varepsilon (x) + A_{k-1}^\varepsilon (n; x) \,\varepsilon^{n-1} \, n!\, x^n \,, \label{3.4}
\end{align}
where polynomials $U_{k}^\varepsilon (x) \,, V_{k-1}^\varepsilon (x)$ and $A_{k-1}^\varepsilon (n; x)$ satisfy the following recurrence relations:
\begin{align}
&\sum_{\ell =1}^{k+1} \binom{k+1}{\ell} \, x^{k-\ell+1} \, U_{\ell}^\varepsilon(x) - \varepsilon \, U_{k}^\varepsilon(x) - x^{k+1} \, = 0 \,, \quad U_{1}^\varepsilon (x) = x - \varepsilon \,, \label{3.5} \\
&\sum_{\ell =1}^{k+1} \binom{k+1}{\ell} \, x^{k-\ell+1} \, V_{\ell-1}^\varepsilon(x) - \varepsilon \, V_{k-1}^\varepsilon(x)\, + \delta_{0k}\, \varepsilon
 x = 0 \,, \quad V_0^\varepsilon (x) = - \varepsilon  \,,   \label{3.6}\\
&\sum_{\ell =1}^{k+1} \binom{k+1}{\ell} \, x^{k-\ell+1} \, A_{\ell-1}^\varepsilon(n;x) - \varepsilon \, A_{k-1}^\varepsilon(n;x) - n^k \, x^{k} \, = 0 \,,    \label{3.7}
\end{align}
where $ A_0^\varepsilon (n; x) = 1$ and  $k \in \mathbb{N}$.
\end{thm}

Proofs for the above theorems can be found in \cite{bd12,bd13}. Recurrence relations \eqref{3.5}-\eqref{3.7} can be taken as definitions of polynomials $U_k^\varepsilon(x) , V_k^\varepsilon (x)$ and $A_k^\varepsilon (n; x)$. Connection of $U_k^\varepsilon(x)$ and  $V_k^\varepsilon (x)$ to $A_k^\varepsilon (n; x)$ follows from the above relations \eqref{3.5}-\eqref{3.7}. Hence, in the sequel
it is sufficient to investigate properties of the generating  polynomials $A_k^\varepsilon (n; x)$, because properties of $U_k^\varepsilon(x)$ and  $V_k^\varepsilon (x)$ can be derived from \eqref{3.2}.

\begin{thm}
Let $A_k^\varepsilon (n; x) = \sum_{j=0}^k A_{kj}^\varepsilon(n) \, x^j$. Then  there is the following  recurrence formula for calculation of $A_{kj}^\varepsilon(n)$:
\begin{align}
&  \binom{k+1}{k+1} A_{kj}^\varepsilon(n) + \binom{k+1}{k} A_{k-1,j-1}^\varepsilon(n) + \binom{k+1}{k-1} A_{k-2,j-2}^\varepsilon(n) + ... \nonumber\\ &+ \binom{k+1}{k-j+1} A_{k-j,0}^\varepsilon(n) \,
   = \begin{cases}  &\varepsilon\, A_{k-1, j}^\varepsilon(n)\, , \quad j = 0, 1, ..., k-1 , \\
      &n^k \, , \qquad \qquad \, \, \,  j = k \, , \end{cases}   \label{3.8}
\end{align}
where $k =  1, 2, 3, ...$ and $A_{00}^\varepsilon (n) = 1$.
\end{thm}

\begin{proof}
Substituting $A_k^\varepsilon (n; x) = \sum_{j=0}^k A_{kj}^\varepsilon(n) \, x^j$ into recurrence relation \eqref{3.7} we have
\begin{align}
\sum_{j=0}^{\ell -1} x^j  \sum_{\ell =1}^{k+1} \binom{k+1}{\ell}  A_{\ell-1,j}^\varepsilon(n) x^{k +1 -\ell}  - \varepsilon  \sum_{j=0}^{k -1} A_{k-1,j}^\varepsilon(n) x^j - n^k  x^{k} \, = 0  \,,  \label{3.9}
\end{align}
where $\quad k = 1,2,3, ... $ Looking for terms with $x^i$ and $x^k$ we obtain
\begin{align}
\bullet\, {x^i:} \quad  &\binom{k+1}{k+1} A_{ki}^\varepsilon(n) + \binom{k+1}{k} A_{k-1,i-1}^\varepsilon(n) + \binom{k+1}{k-1} A_{k-2,i-2}^\varepsilon(n) \nonumber\\& + ... + \binom{k+1}{k-i+1} A_{k-i,0}^\varepsilon(n) - \varepsilon\, A_{k-1, i}^\varepsilon(n) =  0 \,,  \, \, \, i = 0, 1, ..., k-1 \, . \label{3.10}  \\
\bullet \, {x^k:} \quad  &\binom{k+1}{k+1} A_{kk}^\varepsilon(n) + \binom{k+1}{k} A_{k-1,k-1}^\varepsilon(n) + \binom{k+1}{k-1}  A_{k-2,k-2}^\varepsilon(n) \nonumber \\ &+ ... + \binom{k+1}{1} A_{00}^\varepsilon(n) - n^k  =  0 \, .  \label{3.11}  \end{align}
\end{proof}

 Using recurrence relation \eqref{3.8} we can  calculate  $A_{kj}^\varepsilon(n)$. Performing calculation for  $k = 1, 2, ..., 5$, we obtain $A_{k}^\varepsilon(n; x)$ presented below.

\begin{align}
 A_0^\varepsilon (n; x) = &A_{00}^\varepsilon (n) = 1.   \label{3.12} \\
 A_{1}^\varepsilon(n; x) = &\sum_{j =0}^1 A_{1j}^\varepsilon(n) \, x^j = (n  -2) x  + \varepsilon.  \label{3.13} \\
 A_{2}^\varepsilon(n; x) = &\sum_{j =0}^2 A_{2j}^\varepsilon(n) \, x^j = (n^2-3n +3) x^{2} \nonumber \\ &+ (n-5)\varepsilon x  + 1. \label{3.14} \\
  A_{3}^\varepsilon(n; x) = &\sum_{j =0}^3 A_{3j}^\varepsilon(n) \, x^j = (n^3 - 4n^2 +6n-4) x^{3} \nonumber \\
         &+ (n^2 -7n  +17) \varepsilon x^{2} + (n-9) x + \varepsilon. \label{3.15} \\
  A_{4}^\varepsilon(n; x) = &\sum_{j =0}^4 A_{4j}^\varepsilon(n) \, x^j =  (n^4 -5n^3 +10n^2 -10n +5) x^{4} \nonumber \\  &+  (n^3 - 9 n^2 + 31n - 49) \varepsilon x^{3} \nonumber \\ &+ (n^2 -12n + 52) x^{2} + (n - 14) \varepsilon  x + 1. \label{3.16} \\
   A_{5}^\varepsilon(n; x) = &\sum_{j =0}^5 A_{5j}^\varepsilon(n) \, x^j = (n^5 -6 n^4 +15 n^3 - 20 n^2 +15 n  -6) x^5 \nonumber \\ &+ (n^4 - 11n^3
+ 49n^2 - 111n  +129)\varepsilon x^4 \nonumber \\ &+ (n^3 - 15n^2 + 88 n  - 246) x^3 \nonumber \\ &+ (n^2 - 18 n +121)\varepsilon x^2 + (n - 20) x + \varepsilon . \label{3.17}
\end{align}

In the above examples we can note the following properties:
\begin{align}
&\bullet \, \, A_{kj}^\varepsilon(n) =  \sum_{i=0}^j a_{kj,i} \, n^i \, , \quad a_{kj,j} = \varepsilon^{k+j}\,, \quad \varepsilon = \pm 1. \label{3.18} \\
&\bullet \, \, A_{kk}^\varepsilon(n) = \sum_{i=0}^k a_{kk,i}\, n^i  \, , \quad a_{kk,i} = (-1)^{k+i} \left(\begin{aligned} k&+1\\ i&+1\end{aligned}\right) . \label{3.19} \\
&\bullet \, \, A_{k1}^\varepsilon(n) =  \sum_{i=0}^1 a_{k1,i} \, n^i  = \left(n - \frac{k(k+3)}{2}\right) \varepsilon^{k+1} \, , \label{3.20}
\end{align}
where $\frac{k(k+3)}{2}$ in \eqref{3.20} is the sum of $2 +3 + ... + (k+1)$.

Using expressions \eqref{3.12}-\eqref{3.17} we can write some useful values of $A_k^\varepsilon (n; x)$ in Table 2.

\begin{table}
 \begin{center}
 {\begin{tabular}{|c|c|c|c|c|c|c|c|c|c|c|c|}
 \hline \ & \ & \ & \ & \ & \ & \  \\
 k & 1  & 2 & 3 & 4 &  5 & 6  \\
  \hline \  & \  &  \ & \ & \ & \ & \  \\
 $A_k^\varepsilon(0;1)$ &  1  & \ -2 + $\varepsilon$ & 4 - 5$\varepsilon$ & -13 + 18$\varepsilon$ & 58 -63$\varepsilon$ & -272 + 251$\varepsilon$  \\
  \hline \  & \  & \  &  & \ & \ & \  \\
 $A_k^\varepsilon(1;1)$ & 1 & -1+$\varepsilon$ & 2-4$\varepsilon$ & -9 +12 $\varepsilon$ & 43 - 39$\varepsilon$ & -192+162$\varepsilon$  \\
  \hline \  & \  &  \ & \ & \ & \ & \  \\
 $A_k^\varepsilon(0;-1)$ &  1  & \ 2 + $\varepsilon$ & 4 + 5$\varepsilon$ & 13 + 18$\varepsilon$ & 58 +63$\varepsilon$ & 272 + 251$\varepsilon$  \\
  \hline \  & \  & \  &  & \ & \ & \  \\
 $A_k^\varepsilon(1;-1)$ & 1 & 1+$\varepsilon$ & 2+4$\varepsilon$ & 9 +12 $\varepsilon$ & 43 + 39$\varepsilon$ & 192+162$\varepsilon$  \\
 \hline
\end{tabular}}{}
\vspace{4mm}
\caption{ The first six  values of $A_k^\varepsilon(n;x)$, where $n=0,1$ and $x =\pm 1$. } 
\end{center}
\end{table}

Denoting $A_k^{+1} (n; x)\equiv A_k^{+} (n; x)$ and $A_k^{-1} (n; x)\equiv A_k^{-} (n; x)$ from Table 2 we have (cf. \cite{sloane})
\begin{align}
&A_k^{+} (0;1): \, \, \, 1, \, -1, \, -1, \, 5, \, -5, \, -21, \, ...  \quad cf. \, \, \,  A014619  \label {3.21} \\
&A_k^{-} (0;1): \, \, \, 1, \, -3, \, 9, \, -31, \, 121, \, -523, \, ... \quad  cf. \, \, \,  A040027  \label{3.22} \\
&A_k^{+} (1;1): \, \, \, 1, \, 0, \, -2, \, 3, \, 4, \, -30, \, ... \quad  cf. \, \, \,  A007114   \label{3.23} \\
&A_k^{-} (1;1): \, \, \, 1, \, -2, \, 6, \, -21, \, 82, \, -354, \, ...  \quad  cf. \, \, \, A032347 \label{3.24} \\
&A_k^{+} (0;-1): \, \, \, 1, \, 3, \, 9, \, 31, \, 121, \, 523, \, ... \quad  cf. \, \, \, A040027  \label{3.25} \\
&A_k^{-} (0;-1): \, \, \, 1, \, 1, \, -1, \, -5, \, -5, \, 21, \, ... \quad  cf. \, \, \, A014619  \label{3.26} \\
&A_k^{+} (1;-1): \, \, \, 1, \, 2, \, 6, \, 21, \, 82, \, 354, \, ...  \quad  cf. \, \, \, A032347  \label{3.27} \\
&A_k^{-} (1;-1): \, \, \, 1, \, 0, \, -2, \, -3, \, 4, \, 30, \, ...  \quad  cf. \, \, \, A007114  \, . \label{3.28}
\end{align}

 $A_k^\varepsilon (n; x)$ in sequences \eqref{3.21}-\eqref{3.28}  overlap (up to the sign) with related elements in the on-line encyclopedia of integer sequences \cite{sloane}. Using \eqref{3.21}-\eqref{3.28} it is worth to present also $U_k^\varepsilon (x)$ for $x =\pm 1$ and $\varepsilon =\pm 1,$ see
 below.

\begin{align}
&U_k^{+}(1): \, \, \, 0, \, 1, \, -1, \, -2, \, 9, \, -9, \, ... \quad  cf. \, \, \,  A000587  \label{3.29} \\
&U_k^{-}(1): \, \, \, 2, \, -5, \, 15, \, -52, \, 203, \, -877, \, ... \quad  cf. \, \, \, A000110 \label{3.30}  \\
&U_k^{+}(-1): \, \, \, -2, \, -5, \, -15, \, -52, \, -203, \, -877, \, ... \quad  cf. \, \, \, A000110 \label{3.31} \\
&U_k^{-}(-1): \, \, \, 0, \, 1, \, 1, \, -2, \, -9, \, -9, \, ... \quad  cf. \, \, \, A000587 \, .   \label{3.32}
\end{align}

It is worth pointing out that $-U_k^{+}(-1) = B_{k+1} $ for $k \in \mathbb{N},$ where $B_{k}$ are the Bell numbers.
Recall that the Bell number $B_k$ is equal to the number of partitions of a set of $k$ elements. These numbers satisfy the recurrence relation
\begin{align}
B_{k+1} = \sum_{\ell =0}^{k} \binom{k}{\ell}  \, B_{\ell} \,, \qquad B_0 = 1 . \label{3.33}
\end{align}

Taking limit $n \to \infty$ in \eqref{3.4}  and then replacing $i$ by $n$  one obtains
\begin{align}
\sum_{n=0}^\infty n! \,[n^k x^k + U_k^\varepsilon (x)]\, x^n = V_k^\varepsilon (x)   \,.   \label{3.34}
\end{align}
Summation formula \eqref{3.34} can be generalized in
\begin{align}
\sum_{n=0}^\infty n! \,P_k^\varepsilon (n; x)\, x^n = Q_k^\varepsilon (x)   \,,   \label{3.35}
\end{align}
where polynomials $P_k^\varepsilon (n; x)$ and $Q_k^\varepsilon (x)$ are
\begin{align}
P_k^\varepsilon (n; x) = \sum_{j=1}^k C_j \,[n^j x^j + U_j^\varepsilon (x)]\,, \quad Q_k^\varepsilon (x) = \sum_{j=1}^k C_j \, V_j^\varepsilon (x) \label{3.36}
\end{align}
and $ C_j \in \mathbb{Q} .$  Summation formula \eqref{3.36} is the general one, $p$-adic invariant, valid for all $x \in \mathbb{Z}$ and $ Q_k^\varepsilon (x) \in \mathbb{Q} .$

\section{Concluding remarks}

Summarizing, in this paper we presented a brief review of previously obtained particular results on summation of $p$-adic functional series of the form $\sum_{n=1}^\infty n! P_k^\varepsilon (n; x) x^n .$ Some new results are contained in section 3, particularly in formulas \eqref{3.34}-\eqref{3.36}. Significant role of the polynomials $A_k^\varepsilon (n; x)$ is emphasized, because all information on summability of \eqref{3.34} is coded in these polynomials. Moreover, the polynomials $A_k^\varepsilon (n; x)$ for $x =\pm 1$ and $\varepsilon =\pm 1$ are related to some well known
sequences of integer numbers, including the Bell numbers. In \cite{bd13} some aspects of the more general series of the type $\sum_{n=1}^\infty (n + \nu)! P_k^\varepsilon (n; x) x^{\alpha n+ \beta}$ are considered, which can be easily connected to results obtained in section 3.

It is worth mentioning  that power series everywhere convergent on $\mathbb{R}$ and all $\mathbb{Q}_p$, as well as their adelic aspects, are considered in \cite{bd3,bd6}. Kurepa's hypothesis on the left factorial is investigated in \cite{bd10} and infinitely many its equivalents are found. Some concrete examples
can be found in author's papers, in particular in \cite{bd8}.

\section*{Acknowledgements}
This work was supported in part by Ministry of Education, Science and Technological Development of the Republic of Serbia, project OI 174012.
I would like to thank the referee for pointing out  misprints.

\end{document}